\documentclass[11pt]{amsart}

\usepackage{amsmath}
\usepackage{amsthm}
\usepackage{amssymb}
\usepackage[dvipdfmx]{graphicx}

\DeclareMathOperator{\sd}{sd}

\theoremstyle{plain}

\newtheorem{thm}{Theorem}
\newtheorem{cor}{Corollary}
\newtheorem{prop}{Proposition}
\newtheorem{lemma}{Lemma}
\newtheorem*{lemma1'}{Lemma 1'}
\newtheorem*{prop2'}{Proposition 2'}
\newtheorem*{lemma3'}{Lemma 3'}
\newtheorem*{prop4'}{Proposition 4'}
\theoremstyle{definition}

\theoremstyle{remark}
\newtheorem{remark}{Remark}
\theoremstyle{remark}

\newcommand{\vanish}[1]{}

\newcommand{\eulerian}[2]{\left\langle\!\!\begin{array}{c}#1\\#2\end{array}\!\!\right\rangle}

\title[The root distributions of Ehrhart polynomials of free sums]
       {The root distributions of Ehrhart polynomials of 
       free sums of reflexive polytopes}
\date{\empty}

\author{Masahiro Hachimori}
\address{Faculty of Engineering, Information and Systems University of Tsukuba,
Tsukuba, Ibaraki 305-8573, Japan}
\email{hachi@sk.tsukuba.ac.jp}
\author{Akihiro Higashitani}
\address{Department of Pure and Applied Mathematics, 
Graduate School of Information Science and Technology, Osaka University,
Suita, Osaka 565-0871, Japan}
\email{higashitani@ist.osaka-u.ac.jp}
\author{Yumi Yamada}
\address{Department of Policy and Planning Sciences,
Graduate School of Systems and Information Engineering,
Tsukuba, Ibaraki 305-8573, Japan}
\email{s1830124@s.tsukuba.ac.jp}


\date{October 2021}

\begin{document}

\maketitle
\begin{abstract}
In this paper, we study the root distributions of Ehrhart polynomials of free sums of certain reflexive polytopes. 
% We prove that all the roots of the Ehrhart polynomials of the free sums of $A_n^\vee$'s or $A_n$'s 
We investigate cases where the roots of the Ehrhart polynomials of the free sums of $A_d^\vee$'s or $A_d$'s 
lie on the canonical line $\mathrm{Re}(z)=-\frac{1}{2}$ on the complex plane $\mathbb{C}$, 
where $A_d$ denotes the root polytope of type A of dimension $d$ and $A_d^\vee$ denotes its polar dual. 
For example, it is proved that $A_m^\vee \oplus A_n^\vee$ with $\min\{m,n\} \leq 1$ or $m+n \leq 7$, 
$A_2^\vee \oplus (A_1^\vee)^{\oplus n}$ and $A_3^\vee \oplus (A_1^\vee)^{\oplus n}$ for any $n$ satisfy this property. 
We also perform computational experiments for other types of free sums of $A_n^\vee$'s or $A_n$'s. 
\end{abstract}

\section{Introduction}

A polytope $Q\subseteq \mathbb{R}^d$ is called \textit{integral} if all the vertices are on $\mathbb{Z}^d$. 
For an integral polytope $Q \subset \mathbb{R}^d$ of dimension $d$ and a positive integer $k$, 
$E_Q(k)=\#(kQ\cap\mathbb{Z}^d)$ is known to be a polynomial of degree $d$, where $kQ = \{ kx : x\in Q\}$. 
This polynomial is called the \textit{Ehrhart polynomial} of $Q$. 
Its generating function $\mathrm{Ehr}_Q(t)$, called the \textit{Ehrhart series}, can be written as
$$
   \mathrm{Ehr}_Q(t) = \sum_{k=0}^{\infty} E_Q(k)t^k = \frac{\delta_0+\delta_1 t + \dots +\delta_d t^d}{(1-t)^{d+1}},
$$
where the numerator is the \textit{$\delta$-polynomial} of $Q$, denoted by $\delta_Q(t)$, and 
the sequence of the coefficients $\delta(Q)=(\delta_0, \delta_1, \dots, \delta_d)$ is the \textit{$\delta$-vector} of $Q$. 
(They are also known as $h^*$-polynomial and $h^*$-vector, respectively.) 
The $\delta$-vector fully encodes the Ehrhart polynomial and $E_Q(k)$ can be recovered from $\delta(Q)$ as follows:
$$
    E_Q(k) = \sum_{j=0}^d \delta_j \binom{d+k-j}{d} =: f^\mathrm{Ehr}(\delta(Q)) .
$$
We refer the reader to \cite{BR:2007} for the introduction to the Ehrhart polynomials and $\delta$-polynomials of integral polytopes. 

For a polytope $Q \subset \mathbb{R}^d$, the \textit{polar dual} of $Q$ is defined by 
$$
	Q^\vee= \{ x \in \mathbb{R}^d : \langle x,y \rangle \geq -1 \text{ for any }y \in Q\}, 
$$
where $\langle \cdot, \cdot \rangle$ denotes the usual inner product of $\mathbb{R}^d$. 
Note that $(Q^\vee)^\vee=Q$ holds for polytopes containing the origin. 
An integral polytope containing the origin in its interior is \textit{reflexive} 
if its polar dual is also an integral polytope (\cite{B:1994, H:1992}). Note that if $Q$ is reflexive, then so is $Q^\vee$. 
It is known that an integral polytope is reflexive if and only if its $\delta$-vector is palindromic, 
and correspondingly, the roots of the Ehrhart polynomial distribute symmetrically 
with respect to the line ${\rm Re}(z)=-\frac{1}{2}$ on the complex plane $\mathbb{C}$. (See, e.g., \cite[Proposition 2.1]{HKM:2017}.) 
Naturally, it is of interest when the roots of the Ehrhart polynomials all lie on the line ${\rm Re}(z)=-\frac{1}{2}$. Such reflexive polytopes are called ``CL-polytopes'' (\cite{HHK:2019}) 
and studied in several papers (e.g. \cite{HHK:2019, HKM:2017, HY:2021, MHNOH}). 
(In what follows, we call a reflexive polytope \textit{CL} if it is a CL-polytope.)

For two integral polytopes $Q_1\subset \mathbb{R}^{\dim Q_1}$ and 
$Q_2 \subset \mathbb{R}^{\dim Q_2}$, both containing the origins, the free sum $Q_1\oplus Q_2$ is defined by 
$$
Q_1\oplus Q_2 = {\rm conv} ((Q_1\times 0_{Q_2})\cup (0_{Q_1} \times Q_2)) \subset \mathbb{R}^{\dim Q_1}\times \mathbb{R}^{\dim Q_2}, 
$$
where $0_{Q_1}$ and $0_{Q_2}$ are the origins of $\mathbb{R}^{\dim Q_1}$ and $\mathbb{R}^{\dim Q_2}$, respectively. 
The operation of the free sum has importance since it is the polar dual
of the Cartesian product in such a way that $$ (Q_1\times Q_2)^\vee = Q_1^\vee \oplus Q_2^\vee.$$
%where $Q^\vee$ is the polar dual of an integral polytope $Q$.
Note that $Q_1\oplus Q_2$ is reflexive if and only if both $Q_1$ and $Q_2$ are reflexive. 

The $\delta$-polynomial of the free sum has the following simple formula~\cite{B:2006}:
\begin{align}\label{eq:product} \delta_{Q_1\oplus Q_2}(t) = \delta_{Q_1}(t)\delta_{Q_2}(t). \end{align}
On the other hand, the Ehrhart polynomial of $Q_1\oplus Q_2$ can be given (see \cite{B:2006})
but not so simple,
% $$ E_{Q_1\oplus Q_2}(k) = E_{Q_1}(k) + \sum_{j=1}^k E_{Q_2}(k-j)( E_{Q_1}(j)-E_{Q_1}(j-1) ) $$
and the root distribution of the Ehrhart polynomial of $Q_1\oplus Q_2$ is not clear.
Especially, as we will see later,
$Q_1\oplus Q_2$ is not always CL even if both $Q_1$ and $Q_2$ are CL.
In this paper, we are interested in when $Q_1\oplus Q_2$ becomes CL for CL polytopes $Q_1$ and $Q_2$.

A typical example of CL-polytopes is the following special case. 
For a reflexive polytope $Q$, when all the roots $z$ of the $\delta$-polynomial of $Q$ satisfy $|z|=1$, 
it follows from \cite{RV:2002} that all the roots of the Ehrhart polynomial of $Q$ are on the line ${\rm Re}(z)=-\frac{1}{2}$, i.e., $Q$ is CL. 
For example, the following polytopes are such examples: 
\begin{itemize}
\item A cross polytope $\mathrm{Cr}_d={\rm conv}(\{e_1,\dots,e_d,-e_1,\dots,-e_d \})$, where
      $\delta_{\mathrm{Cr}_d}(t)=(1+t)^d$.
\item A simplex $T_d = {\rm conv}(\{e_1,\dots,e_d,-(e_1+\dots+e_d) \})$, where
      $\delta_{T_d}(t)=1+t+\dots+t^d$.
\end{itemize}
Here, $e_i$ denotes the $i$-th unit vector of $\mathbb{R}^d$. 
If $Q_i$'s are such polytopes, then we have $Q_1\oplus \dots\oplus Q_n$ is CL 
since the roots of the $\delta$-polynomial of $Q_1\oplus \dots\oplus Q_n$ also satisfy $|z|=1$ by \eqref{eq:product}. 
(Notice that $\mathrm{Cr}_d$ is unimodularly equivalent to $\underbrace{T_1 \oplus \cdots \oplus T_1}_d$.) 

In this paper, we mainly discuss the case $Q_i$'s are the dual of the classical root polytopes of type A.
Here, the classical root polytope of type A is defined as
$$ A_d = {\rm conv}(\{ \pm(e_i + \dots + e_j) : 1\le i\le j\le d\}), $$
and we consider its dual $A^\vee_d$. 
The Ehrhart polynomial of $A^\vee_d$ is known to be $$ E_{A^\vee_d}(k)= (k+1)^{d+1} - k^{d+1} $$ 
in \cite[Lemma 5.3]{HKM:2017}. 
Reflexive polytopes $A_d$ and $A_d^\vee$ are shown to be CL in \cite{HKM:2017}, 
but we see the roots $z$ of their $\delta$-polynomials do not satisfy $|z|=1$. 
The reason we consider this free sum of $A^\vee_d$'s is that 
it appears as the equatorial spheres of the complete graded posets.
This will be discussed in Section~\ref{sec:posets}.
% After that, we investigate the CL-ness of $A^\vee_{p_1,p_2,\dots,p_k}$ 
After that, we investigate the CL-ness of $A^\vee_{p_1}\oplus A^\vee_{p_2}\oplus\dots\oplus A^\vee_{p_k}$ 
in the following sections.

We collect the results which show the CL-ness for the free sums of $A_d^\vee$'s or $A_d$'s in what follows: 
\begin{itemize}
\item $A_m^\vee \oplus A_n^\vee$ with $\min\{m,n\} \leq 1$ or $m+n \leq 7$ (Theorem~\ref{thm:rank2cp}); 
\item $A_m^\vee \oplus (A_1^\vee)^{\oplus n}$ for any $n \geq 1$ with $m=1,2,3$ (Proposition~\ref{thm:1111}, Theorems~\ref{thm:2111} and \ref{thm:3111}); 
\item $A_1 \oplus A_n$ for any $n \geq 1$ (Theorem~\ref{thm:A1n}); 
\item $A_m \oplus A_1^{\oplus n}$ for any $n \geq 1$ with $m=1,2,3$ (Proposition~\ref{thm:A111-211} and Theorem~\ref{thm:A3111}). 
\end{itemize}
We also perform other types of free sums of $A_n^\vee$'s or $A_n$'s and describe the computational results.  

\bigskip
\section*{Acknowledgments}
The second named author is partially supported by JSPS Grant-in-Aid for Scientists Research (C) 20K03513.

\bigskip
%%%%%%%%%%%%%%%%%%%%%%%%%%%%%%%%%%%%%%%%%%%%%%%%%%%%%%%%%%%%%%%%%%%%%%%%%%%%%%%%%%%%%%%%%%%%%%%%%%%%%%%%%%%%%%%%%%%%%%%%%%%%%%%%%%%%%%%%
%%%%%%%%%%%%%%%%%%%%%%%%%%%%%%%%%%%%%%%%%%%%%%%%%%%%%%%%%%%%%%%%%%%%%%%%%%%%%%%%%%%%%%%%%%%%%%%%%%%%%%%%%%%%%%%%%%%%%%%%%%%%%%%%%%%%%%%%
%%%%%%%%%%%%%%%%%%%%%%%%%%%%%%%%%%%%%%%%%%%%%%%%%%%%%%%%%%%%%%%%%%%%%%%%%%%%%%%%%%%%%%%%%%%%%%%%%%%%%%%%%%%%%%%%%%%%%%%%%%%%%%%%%%%%%%%%
%%%%%%%%%%%%%%%%%%%%%%%%%%%%%%%%%%%%%%%%%%%%%%%%%%%%%%%%%%%%%%%%%%%%%%%%%%%%%%%%%%%%%%%%%%%%%%%%%%%%%%%%%%%%%%%%%%%%%%%%%%%%%%%%%%%%%%%%
%%%%%%%%%%%%%%%%%%%%%%%%%%%%%%%%%%%%%%%%%%%%%%%%%%%%%%%%%%%%%%%%%%%%%%%%%%%%%%%%%%%%%%%%%%%%%%%%%%%%%%%%%%%%%%%%%%%%%%%%%%%%%%%%%%%%%%%%

\section{Ehrhart polynomials of equatorial spheres of graded posets}\label{sec:posets}

Let $(P, \preceq)$ be a finite partially ordered set, or a poset, with $|P|=d$. 
The \textit{order polytope} $O_P$ of $P$ is given by 
$$
     O_P = \{ x\in \mathbb{R}^d : x_a\le x_b \; \text{ for }\; b \prec a \;(a,b\in P)\; \},
$$
where the coordinates of $\mathbb{R}^d$ are indexed by the elements of $P$. 
This is an integral polytope whose vertices correspond to the order ideals of $P$ (\cite{S:1986}).

As another polytope arising from posets closely related to the order polytope, the \textit{chain polytope} of $P$ is defined by 
\begin{align*}
C_P = \{  x\in \mathbb{R}^d :\;  x_a \ge 0  \; (a\in P), \;\; 
x_{a_1}+\dots+x_{a_k}\le 1  \;\text{ for }\; a_1 \prec \cdots \prec a_k \; (a_i \in P)\; \}. 
\end{align*}
This is an integral polytope whose vertices correspond to the antichains of $P$,
and it is shown in \cite{S:1986} that the Ehrhart polynomials of $O_P$ and $C_P$ coincide: $E_{O_P}(k) = E_{C_P}(k)$, 
so we also have $\delta_{O_P}(t)=\delta_{C_P}(t)$. 

\medskip
For the poset $P$ on $[n] = \{1,2,\dots,n\}$, the \textit{$P$-Eulerian polynomial} is 
$$
    W(P) = \sum_{\pi\in {\mathcal L}(P)} x^{{\rm des}(\pi)+1},
$$
where ${\mathcal L}(P)$ is the set of all linear extensions of $P$ and ${\rm des}(\pi)$ is the size of the descent set of $w$ with respect to $P$. 
That is, ${\mathcal L}(P)$ is the set of permutations $w=(w_1,w_2,\dots,w_n)$ of $[n]$ such that $w_i\prec w_j$ implies $i<j$, and ${\rm des}(w)=\# \{i\in [n-1] : w_i > w_{i+1} \}$. 
This polynomial is equal to the $\delta$-polynomial of $O_P$, i.e., $W(P)=\delta_{O_P}(t)$. 
%This polynomial is closely related to the Ehrhart polynomial in such a way that $$    W(P) = \delta_0+\delta_1 t + \dots +\delta_d t^d = \delta_{O_P}(t). $$

When the poset $P$ is graded of rank $r$, the result of \cite{RW:2005} shows that the $\delta$-vector can be written as
$$ 
   \delta(O_P) = h(\Delta_\text{eq}(P)*\sigma^r), 
$$
where $\sigma^r$ is the $r$-dimensional standard simplex and $\Delta_\text{eq}(P)$ is the equatorial sphere of $P$, which will be explained below. Here, 
% $\Delta_\text{eq}(P)$ and $\sigma^r$ are simplicial complexes,
the operator $*$ is the simplicial join of simplicial complexes 
and $h(\Delta_\text{eq}(P)*\sigma^r)$ represents the $h$-vector of the simplicial complex $\Delta_\text{eq}(P)*\sigma^r$.

For a poset $P$, a \textit{$P$-partition} is a function $f: P\rightarrow \mathbb{R}$
such that $f(a)\ge 0$ for all $a\in P$ and $f(a)\geq f(b)$ for all $a\prec b$.
When $P$ is a graded poset of rank $r$, let $P^{(i)}$ denote the set of the elements of $P$ of rank $i$.
We say that a $P$-partition is \textit{equatorial} if $\min_{a\in P} f(a)=0$ and 
for every $2 \leq j \leq r$ there exists $a_{j-1}\prec a_j$ with $a_{j-1}\in P^{(j-1)}$, $a_j\in P^{(j)}$ and $f(a_{j-1})=f(a_j)$.
An order ideal $I$ of $P$ is \textit{equatorial} if its characteristic vector $\chi_I$ is equatorial.
A chain of order ideals $I_1\subset I_2\subset \dots\subset I_t$ is \textit{equatorial}
if $\chi_{I_i}+\dots+\chi_{I_t}$ is equatorial.
The \textit{equatorial complex} $\Delta_\text{eq}(P)$ of $P$ is the simplicial complex
whose  vertex set is the equatorial ideals of $P$ and faces are 
equatorial chains of order ideals of $P$.
The result of \cite{RW:2005} shows that $\Delta_\text{eq}(P)$ is a (polytopal) simplicial sphere
and it is called the \textit{equatorial sphere} of $P$.
Since the $h$-vector of a simplicial sphere is palindromic by the Dehn-Sommerville equations, 
this implies that the $\delta$-vector of $O_P$ for a graded poset $P$ is palindromic followed by $r$ 0's as follows: 
$$
   \delta(O_P) = (h_0, h_1, \dots, h_1, h_0, \underbrace{0, 0, \dots, 0}_r). 
$$
The palindromic part $(h_0,h_1,\dots,h_1,h_0)=h(\Delta_\text{eq})$ of $\delta(O_P)$ corresponds to 
the equatorial sphere $\Delta_\text{eq}(P)$, so it will make sense to consider the 
corresponding polynomial as follows.
$$
  E^\text{eq}_{P}(k) = f^\mathrm{Ehr}( h(\Delta_\text{eq}) ) = f^\mathrm{Ehr}((h_0, h_1, \dots, h_1, h_0)).
$$
We call this $E^\text{eq}_{P}(k)$ the \textit{equatorial Ehrhart polynomial} of the graded poset $P$.
In \cite{RW:2005}, the equatorial sphere is constructed as a quotient polytope 
from the order polytope, that is, as a quotient polytope $O^\text{eq}_P=O_P/V^\text{rc}$, where 
$V^\text{rc}$ is the rank-constant subspace, the subspace consisting of partition functions
that are rank-constant (i.e., $f(x)=f(y)$ whenever $x$ and $y$ are of the same rank in $P$).
The polynomial $E^\text{eq}_{P}(k)$ corresponds to the Ehrhart polynomial of this polytope. 

Since $h(\Delta_\text{eq})$ is palindromic, the roots of $E^\text{eq}_{P}(k)$ distribute symmetrically with respect to the line ${\rm Re}(z)=-\frac{1}{2}$.
It is of our interest for which graded poset $P$ all the roots of $E^\text{eq}_{P}(k)$ lie on the line ${\rm Re}(z)=-\frac{1}{2}$.
We call such $E^\text{eq}_{P}(k)$ to be CL analogously to the CL-polytopes among reflexive polytopes.

\medskip
A \textit{complete graded poset} $P_{n_1,n_2,\dots,n_r}$ stands for a graded poset of rank $r$ such that 
the set $P^{(i)}$ of the elements of rank $i$ consists of $n_i$ elements for every $i$ and $a_i\prec a_j$ holds for every $a_i\in P^{(i)}$ and $a_j\in P^{(j)}$ with $i<j$. 
For complete graded posets, we can easily calculate the $\delta$-polynomials as follows.
Since the antichains of $P_{n_1,n_2,\dots,n_r}$ are a subset $X\subset P^{(i)}$ for some $i$, we have 
$$C_{P_{n_1,n_2,\dots,n_r}} = [0,1]^{n_1}\oplus [0,1]^{n_2}\oplus\dots\oplus [0,1]^{n_r}. $$
The $\delta$-polynomial of $[0,1]^{n}$ is given by the Eulerian polynomial 
$S_{n}(t) = \sum_{j=0}^{n-1}\eulerian{n}{j}t^j$,
where $\eulerian{n}{j}$ is the Eulerian number,
and hence we have 
$$ \delta_{O_{P_{n_1,n_2,\dots,n_r}}}(t) = \delta_{C_{P_{n_1,n_2,\dots,n_r}}}(t) = \prod_{i=1}^r S_{n_i}(t). $$

There is another explanation for this.
For the complete graded poset $P_{n_1,n_2,\dots,n_r}$, 
an equatorial ideal is a proper subset of $P^{(i)}$ for some $0\le i\le r$ together with 
all $P^{(j)}$'s with $j<i$, 
hence we observe that 
$\Delta_\text{eq}(P_{n_1,n_2,\dots,n_r})$ is isomorphic 
to the order complex of $\check{B}_{n_1}\biguplus \dots \biguplus \check{B}_{n_r}$, 
where $\check{B}_n$ is the poset removing the top element from the boolean lattice 
of order $n$ ($=$ the ordered set consisting of all the strict subsets of $\{1,\dots,n\}$
ordered by inclusion), and $\biguplus$ is the operator of the ordinal sum of the posets
(i.e., $P\biguplus P'$ is the poset over $P\cup P'$ with an order relation $\preceq_{P\biguplus P'}$ 
such that $u\preceq_{P\biguplus P'} v$ if $u,v\in P$ and $u\preceq_{P} v$, $u,v\in P'$ and $u\preceq_{P'} v$, or $u\in P$ and $v\in P'$). 
This shows that the equatorial sphere of $P_{n_1,n_2,\dots,n_r}$ is isomorphic to $\sd(\Delta_{n_1}) * \dots * \sd(\Delta_{n_r})$.  
%where the operator $*$ is the simplicial join.
Since the $h$-polynomial of $\sd(\Delta_{n})$ is given by the Eulerian polynomial (see, e.g., \cite[Sec.~9.2]{P:2015}), we have the same conclusion. 

The equatorial Ehrhart polynomial for $P_n$, which is just an antichain with $n$ elements, can be calculated as follows. Since we have 
$\delta_i = \eulerian{n}{i}$ for $ 0\le i\le n-1$, where $\delta(O_{P_{n}})=(\delta_0,\delta_1,\ldots,\delta_n)$, we obtain that 
\begin{align*}
  E^\mathrm{eq}_{P_{n}}(k) &= \sum_{j=0}^{n-1} \eulerian{n}{j}\binom{n-1+k-j}{n-1} 
                     = \sum_{j=0}^{n-1} \eulerian{n}{n-1-j}\binom{k+(n-1-j)}{n-1} \\
                    &= \sum_{j'=0}^{n-1} \eulerian{n}{j'}\binom{k+j'}{n-1}  \quad (j'=n-1-j) \\
         &= \sum_{j'=0}^{n-1} \eulerian{n}{j'} \left( \binom{k+j'+1}{n} -  \binom{k+j'}{n} \right) 
         = (k+1)^{n} - k^{n} .
\end{align*}
Here, the last equality is derived from Worpitzky's identity (e.g. \cite[Sec.~6.2]{GKP:1994}):
$\displaystyle x^n = \sum^{n-1}_{j=0} \eulerian{n}{j}\binom{x+j}{n} .$
This polynomial $(k+1)^{n} - k^{n}$ equals the Ehrhart polynomial of $A^\vee_{n-1}$ as shown in \cite{HKM:2017}. 
%where $A^\vee_{n-1}$ is the dual of the classical root polytope $A_{n-1}$.
That is, we have $E^\text{eq}_{P_n}(k) = E_{A^\vee_{n-1}}(k)$. 
In fact, more strongly, we observe that the equatorial polytope $O^\text{eq}_{P_n}$ is unimodularly equivalent to $A^\vee_{n-1}$ as follows. 
%and the fact $E^\text{eq}_{P_n}(k) = E_{A^\vee_{n-1}}(k)$ follows from that. 
\begin{prop}
$O^\text{eq}_{P_n} = O(P_n)/V^\text{rc}$ is unimodularly equivalent to $A^\vee_{n-1}$.
\end{prop}
\begin{proof}
The subspace $V^\text{rc}$ is the space of rank-constant partitions, and in this case, it is a one-dimensional space $V^\text{rc} = {\rm span} \{ \sum_{i\in [n]} e_i \}$. 
Let $\pi$ be the projection map from $O(P_n)$ to $O^\text{eq}_{P_n}$. 
By letting $f=\sum_{i\in [n]} e_i$, for any $v \in \mathbb{R}^n$, we can uniquely write $v=\sum_{i=1}^{n-1} r_i e_i + sf \in V \, (r_i, s\in \mathbb{R})$, then we have $\pi(v) = \sum_{i=1}^{n-1} r_i e_i$. 
The vertex set of $O_{P_n}$ is $\{\sum_{i\in S}e_i : S\subseteq [n] \}$, and they are mapped to the following:
$$
\pi\left( \sum_{i\in S}e_i \right) = 
\begin{cases}
\sum_{i\in S}e_i & \text{if $n \not\in S$}, \\
-\sum_{i\not\in S}e_i & \text{if $n \in S$}.
\end{cases}
$$
From this, we observe that the vertex set of $O^\text{eq}_{P_n}$ is
$\{ \pm \sum_{i\in S} e_i : S\subseteq [n-1] \}$.
Hence $$(O^\text{eq}_{P_n})^\vee = \left\{x \in \mathbb{R}^{n-1}: \left\langle \pm \sum_{i\in S} e_i, x \right\rangle \le 1, \; S \subset [n-1] \right\}.$$

On the one hand, it is easy to see that $A_n$ is unimodularly equivalent to 
\begin{align*}
{\rm conv}(\{\pm e_i : 1\le i\le n-1\} \cup \{e_i-e_j : 1\le i\neq j\le n\}). 
\end{align*}
Since we have 
$$
\left\langle \sum_{i\in S} e_i, \pm e_j \right\rangle = 
\begin{cases}
\pm 1 &\text{if }j\in S, \\
0 &\text{if }j\not\in S, 
\end{cases}
\quad
\text{and} 
\quad
\left\langle \sum_{i\in S} e_i, e_j-e_k \right\rangle = 
\begin{cases}
1 &\text{if }j\in S, k\not\in S, \\
-1 &\text{if }j\not\in S, k\in S, \\
0 &\text{if }j,k \in S \text{ or }j,k \not\in S,  
\end{cases}
$$
we see that $A_n \subseteq (O^\text{eq}_{O_P})^\vee.$
On the other hand, let $w=(w_1,\ldots,w_{n-1}) \in \mathbb{Z}^{n-1}$ satisfying that $\langle w,v \rangle \leq 1$ for any $v \in A_n$. 
If there is $i$ with $|w_i| \geq 2$, then $|\langle w, e_i \rangle| \geq 2$, a contradiction. Thus, $w \in \{0,\pm 1\}^{n-1}$. 
Moreover, if there are $i$ and $i'$ with $w_i =1$ and $w_{i'}=-1$, then $\langle w, e_i-e_{i'}\rangle =2$, a contradiction. 
Hence, $w \in \{0,1\}^{n-1}$ or $w \in \{0,-1\}^{n-1}$. This means that $w$ is always of the form $w=\pm\sum_{i \in S}e_i$. 
This implies that $(O^\text{eq}_{O_P})^\vee \subset A_n$, as required. 
\end{proof}

\begin{cor}
% $$ E^\text{eq}_{P_{n_1,n_2,\dots,n_r}}(k)=E_{A^\vee_{(n_1-1),(n_2-1),\dots,(n_r-1)}}(k).$$
We have $$ E^\text{eq}_{P_{n_1,n_2,\dots,n_r}}(k)=E_{A^\vee_{n_1-1}\oplus A^\vee_{n_2-1}\oplus\dots\oplus A^\vee_{n_r-1}}(k).$$
\end{cor}

By this, the CL-ness of $E^\text{eq}_{P_{n_1,n_2,\dots,n_r}}(k)$ is
% equivalent to the CL-ness of $A^\vee_{n_1,n_2,\dots,n_r}$.
equivalent to the CL-ness of $A^\vee_{n_1-1}\oplus A^\vee_{n_2-1}\oplus\dots\oplus A^\vee_{n_r-1}$.

\begin{remark}
The discussion of this section gives that the $\delta$-polynomial of $A^\vee_d$ equals to
$$ \delta_{A^\vee_d}(t) = \sum_{j=0}^d \eulerian{d+1}{j} t^j .$$
\end{remark}

\section{CL-ness of $A^\vee_m \oplus A^\vee_n$} \label{sec:mn}

% For the case of two parameters $A^\vee_n\oplus A^\vee_m$, we have the following.
For the case of the free sum $A^\vee_1\oplus A^\vee_n$, we have the following.

\begin{prop} \label{prop:P2n}
We have
% $$ E_{A^\vee_{1,n}}(k) = (k+1)^n + k^n,$$
% and $A^\vee_{1,n}$ is a CL-polytope.
$$ E_{A^\vee_1\oplus A^\vee_n}(k) = (k+1)^n + k^n,$$
and $A^\vee_1\oplus A^\vee_n$ is a CL-polytope.
\end{prop}
\begin{proof}
Since $\delta_{A_1^\vee}(t)=1+t$ and $\delta_{A_n^\vee}=\sum_{i=1}^n\eulerian{n+1}{i}t^i$, we have 
$$
%   \delta_i(A^\vee_{1,n}) = \eulerian{n+1}{i}+\eulerian{n+1}{i-1} \quad (0\le i\le n+1) 
  \delta_i(A^\vee_1\oplus A^\vee_n) = \eulerian{n+1}{i}+\eulerian{n+1}{i-1} \quad (0\le i\le n+1) 
$$
using the convention that $\eulerian{n}{i}=0$ when $i<0$ or $i\ge n$. Thus, 
\begin{align*}
%   E_{A^\vee_{1,n}}(k) &= \sum_{j=0}^{n+1} (\eulerian{n+1}{j} + \eulerian{n+1}{j-1})\binom{n+1+k-j}{n+1}  \\
  E_{A^\vee_1\oplus A^\vee_n}(k) &= \sum_{j=0}^{n+1} \left(\eulerian{n+1}{j} + \eulerian{n+1}{j-1}\right)\binom{n+1+k-j}{n+1}  \\
                      &= \sum_{j=0}^{n} \eulerian{n+1}{j}\binom{n+1+k-j}{n+1}
                        +\sum_{j=1}^{n+1} \eulerian{n+1}{j-1}\binom{n+1+k-j}{n+1} \\
                      &= \sum_{j=0}^{n} \eulerian{n+1}{n-j}\binom{1+k+(n-j)}{n+1}
                        +\sum_{j=1}^{n+1} \eulerian{n+1}{n-j+1}\binom{k+(n-j+1)}{n+1} \\
                      &= \sum_{j'=0}^{n} \eulerian{n+1}{j'}\binom{1+k+j'}{n+1}
                        +\sum_{j''=0}^{n} \eulerian{n+1}{j''}\binom{k+j''}{n+1} 
                        \quad \begin{array}{l}(j'=n-j, \\ \: j''=n-j+1)\end{array}
                        \\
                      &= (k+1)^{n+1} + k^{n+1} .
\end{align*}
Here, the last equality is derived by Worpitzky's identity.

This polynomial $(k+1)^{n+1} + k^{n+1}$ equals the Ehrhart polynomial of the polar dual $C_{n+1}^\vee$ of the classical root polytope of type C 
and it is shown to be CL in \cite{HY:2021}. 
\end{proof}

% This theorem shows $A^\vee_{1,n} = A^\vee_1\oplus A^\vee_n= (A_1\times A_n)^\vee$ and $C_{n+1}^\vee$ have
This theorem shows $A^\vee_1\oplus A^\vee_n= (A_1\times A_n)^\vee$ and $C_{n+1}^\vee$ have
the same Ehrhart polynomial, though $A_1\times A_n$ and $C_{n+1}$ are not unimodularly equivalent since $C_{n+1}$ does not have the structure of the product of two polytopes.

\bigskip

The CL-ness of $A^\vee_m \oplus A^\vee_n$ with small $m$ and $n$ are
calculated by computer using Pari/GP. See appendix for the detail.
The results are summarized as shown in Table~\ref{table:Adual-mn}.
From the table, we have the following theorem.

\begin{table}[htb]
\caption{CL-ness of $A^\vee_m \oplus A^\vee_n$ with $m,n\le 20$}
\label{table:Adual-mn}
\begin{center}
\begin{tabular}{| l | l | l | l | l | l | l | l | l | l |} \hline
$m$ $\backslash$ $n$ & 0 & 1 & 2 & 3 & 4 & 5 & 6 & $7\sim 20$ & $\ge 21$ \\ \hline
0 & CL & CL & CL & CL & CL & CL & CL & CL & CL  \\ \hline 
1 & CL & CL & CL & CL & CL & CL & CL & CL & CL \\ \hline 
2 & CL & CL & CL & CL & CL & CL & not CL & not CL  & \\ \hline 
3 & CL & CL & CL & CL & CL & not CL & not CL & not CL  & \\ \hline
4 & CL & CL & CL & CL & not CL & not CL & not CL & not CL  & \\ \hline
5 & CL & CL & CL & not CL & not CL & not CL & not CL & not CL  & \\ \hline
6 & CL & CL & not CL & not CL & not CL & not CL & not CL & not CL  & \\ \hline
$7 \sim 20$ & CL & CL & not CL & not CL & not CL & not CL & not CL  & not CL & \\ \hline
$\ge 21$  & CL & CL &  &  &  &  &  &  &  \\ \hline 
\end{tabular}
\end{center}
\end{table}

\begin{thm} \label{thm:rank2cp}
% $A^\vee_{n,m}$ is CL if $n\le 1$, $m\le 1$, or $n+m\le 7$.
$A^\vee_m \oplus A^\vee_n$ is CL if $\min\{m,n\} \le 1$ or $m+n \le 7$.
\end{thm}

It is not yet shown whether 
all the cases $m\ge 2$ and $n\ge 8$ (or vice versa) are not CL, 
though it is plausible that Theorem~\ref{thm:rank2cp} is also necessary
for $A^\vee_m \oplus A^\vee_n$ to be CL. 
By our computer calculation up to $n,m\le 20$, no other CL parameters are found other than shown above.

\bigskip
%%%%%%%%%%%%%%%%%%%%%%%%%%%%%%%%%%%%%%%%%%%%%%%%%%%%%%%%%%%%%%%%%%%%%%%%%%%%%%%%%%%%%%%%%%%%%%%%%%%%%%%%%%%%%%%%%%%%%%%%%%%%%%%%%%%%%%%%
%%%%%%%%%%%%%%%%%%%%%%%%%%%%%%%%%%%%%%%%%%%%%%%%%%%%%%%%%%%%%%%%%%%%%%%%%%%%%%%%%%%%%%%%%%%%%%%%%%%%%%%%%%%%%%%%%%%%%%%%%%%%%%%%%%%%%%%%
%%%%%%%%%%%%%%%%%%%%%%%%%%%%%%%%%%%%%%%%%%%%%%%%%%%%%%%%%%%%%%%%%%%%%%%%%%%%%%%%%%%%%%%%%%%%%%%%%%%%%%%%%%%%%%%%%%%%%%%%%%%%%%%%%%%%%%%%
%%%%%%%%%%%%%%%%%%%%%%%%%%%%%%%%%%%%%%%%%%%%%%%%%%%%%%%%%%%%%%%%%%%%%%%%%%%%%%%%%%%%%%%%%%%%%%%%%%%%%%%%%%%%%%%%%%%%%%%%%%%%%%%%%%%%%%%%
%%%%%%%%%%%%%%%%%%%%%%%%%%%%%%%%%%%%%%%%%%%%%%%%%%%%%%%%%%%%%%%%%%%%%%%%%%%%%%%%%%%%%%%%%%%%%%%%%%%%%%%%%%%%%%%%%%%%%%%%%%%%%%%%%%%%%%%%

\section{CL-ness of $A^\vee_{n_1}\oplus A^\vee_{n_2}\oplus\dots\oplus A^\vee_{n_r}$}

In the following theorems, we have families of $A^\vee_{p_1}\oplus A^\vee_{p_2}\oplus\dots\oplus A^\vee_{p_r}$ that are CL.
In what follows, we denote $\underbrace{A^\vee_p\oplus A^\vee_p\oplus\dots\oplus A^\vee_p}_{n}$ as $({A^\vee_p})^{\oplus n}$.

\begin{prop}[{\cite[Example 3.3]{HKM:2017}}]\label{thm:1111}
$({A^\vee_1})^{\oplus n}$ is CL for any $n$.
\end{prop}
\begin{proof}
This ${A^\vee_1}\,^{\oplus n}$ is the $n$-dimensional cross polytope $\mathrm{Cr}_n$, and is shown to be CL in \cite[Example 3.3]{HKM:2017}. 
\end{proof}

% We can further show that $A^\vee_{21^m}$ and $A^\vee_{3^m}$ are also CL.
We can further show that $A^\vee_2\oplus (A^\vee_1)^{\oplus n}$ and $A^\vee_3\oplus (A^\vee_1)^{\oplus n}$ are also CL.
For these, we use the following lemma. Here, $R$ is the canonical line $\mathrm{Re}(z)=-1/2$, and 
two functions $f(x)$ and $g(x)$ with $\deg f = \deg g + 1$ are \textit{$R$-interlacing} 
if all the zeros of $f(x)$ and $g(x)$ are on $R$ and they appear alternatingly on $R$. 
That is, the zeros of $f$ are $-1/2+z_1 i, -1/2+z_2 i, \dots, -1/2+z_d i$ and those of $g$ are $-1/2+w_1 i, -1/2+w_2 i, \dots, -1/2+w_{d-1} i$, 
with $z_1< w_1 < z_2 < w_2 < \dots <w_{d-1} < z_d$, where $d=\deg f$.  

\begin{lemma}[{\cite[Lemma 2.5]{HKM:2017}}] \label{lemma:CLrelation}
Let $f_1, f_2$, and $f_3$ be real monic polynomials such that
$\deg f_1 = \deg f_2 +1 = \deg f_3 +2$ and $f_1(x) = f_2(x)\cdot(x+\frac{1}{2}) + \beta f_3(x)$ for some $\beta>0$.
Then $f_1$ and $f_2$ are $R$-interlacing if and only if $f_2$ and $f_3$ are $R$-interlacing. 
\end{lemma}

Note that, when we use this lemma for three Ehrhart polynomials $E_1, E_2$, and $E_3$, 
the relation in the lemma should be 
$$ E_1(k) = \alpha E_2(k)\cdot (2k+1) + (1-\alpha)E_3(k)\; \text{ for some } \; 0\le \alpha\le 1.$$ 
See \cite[Section 3]{HKM:2017}. 
%since the constant term of Ehrhart polynomials equals one.

\begin{thm}\label{thm:2111}
$A^\vee_2\oplus (A^\vee_1)^{\oplus n}$ is CL for any $n$.
\end{thm}
\begin{proof}
We have the following equality: 
\begin{equation}
  E_{A^\vee_2\oplus (A^\vee_1)^{\oplus n}}(k) = \frac{3}{2n+4}E_{(A^\vee_1)^{\oplus (n+1)}}(k)\cdot (2k+1)  + \frac{2n+1}{2n+4} E_{(A^\vee_1)^{\oplus n}}(k). 
\label{eq:eq1}
\end{equation}
This follows from the relation of the Ehrhart series:
\begin{equation}
  \mathrm{Ehr}_{A^\vee_2\oplus (A^\vee_1)^{\oplus n}}(t) = \frac{3}{2n+4} \left( 2t\frac{d}{dt}\mathrm{Ehr}_{(A^\vee_1)^{\oplus (n+1)}}(t) 
                           + \mathrm{Ehr}_{(A^\vee_1)^{\oplus (n+1)}}(t) \right)
                         + \frac{2n+1}{2n+4} \mathrm{Ehr}_{(A^\vee_1)^{\oplus n}}(t).
\label{eqn:1ehr}
\end{equation}
The equation \eqref{eq:eq1} is derived by comparing the coefficients of $t^k$ in \eqref{eqn:1ehr}.
The equation \eqref{eqn:1ehr} can be verified using 
$\mathrm{Ehr}_{A^\vee_2\oplus (A^\vee_1)^{\oplus n}}(t)=\frac{(1+4t+t^2)(t+1)^n}{(1-t)^{n+3}}$ and
$\mathrm{Ehr}_{(A^\vee_1)^{\oplus n}}(t)=\frac{(1+t)^n}{(1-t)^{n+1}}$ as follows:
\begin{align}
\text{RHS of \eqref{eqn:1ehr}} =& 
\frac{3}{2n+4}\bigg(
2t \frac{d}{dt} \frac{(1+t)^{n+1}}{(1-t)^{n+2}} + \frac{(1+t)^{n+1}}{(1-t)^{n+2}} \bigg)
+ \frac{2n+1}{2n+4} \frac{(1+t)^n}{(1-t)^{n+1}} \nonumber \\
=&
\frac{(1+t)^n}{(1-t)^{n+3}}\bigg(
2t\frac{3}{2n+4}\big( (n+1)(1-t) + (1+t)(n+2) \big)  \nonumber \\
& \hspace{50mm} + \frac{3(1+t)(1-t)}{2n+4}
+ \frac{(2n+1)(1-t)^2}{2n+4} \bigg) \nonumber \\
=& \frac{(1+t)^n(1+4t+t^2)}{(1-t)^{n+3}} = \textrm{Ehr}_{A^\vee_2\oplus (A^\vee_1)^{\oplus n}}(t). \nonumber
\end{align}

Since the Ehrhart polynomials of $(A^\vee_1)^{\oplus (n+1)}$ and $(A^\vee_1)^{\oplus n}$ 
(i.e., the cross polytopes $\mathrm{Cr}_{n+1}$ and $\mathrm{Cr}_{n}$) are $R$-interlacing as shown in \cite[Corollary 5.4]{HKM:2017}, 
$A^\vee_2\oplus (A^\vee_1)^{\oplus n}$ and $(A^\vee_1)^{\oplus (n+1)}$ are $R$-interlacing by Lemma~\ref{lemma:CLrelation}. 
Hence, we conclude that $A^\vee_2\oplus (A^\vee_1)^{\oplus n}$ is CL.

\end{proof}

\begin{thm}\label{thm:3111}
% $A^\vee_{31^m}$ is CL for any $m$.
$A^\vee_3\oplus (A^\vee_1)^{\oplus n}$ is CL for any $n$.
\end{thm}
\begin{proof}
We have the following equality: 
\begin{equation}
%   E_{A^\vee_{3 1^m}}(k) = \frac{3}{m+3}E_{A^\vee_{1^{m+2}}}(k)\cdot(2k+1)  + \frac{m}{m+3} E_{A^\vee_{1^{m+1}}}(k). 
  E_{A^\vee_3 \oplus (A^\vee_1)^{\oplus n}}(k) = \frac{3}{n+3}E_{(A^\vee_1)^{\oplus(n+2)}}(k)\cdot(2k+1)  
  + \frac{n}{n+3} E_{(A^\vee_1)^{\oplus (n+1)}}(k). 
\label{eq:eq2}
\end{equation}

% This equation can be shown in the same way as in Theorem~\ref{thm:2111} with 
% %$\mathrm{Ehr}_{A^\vee_3\oplus (A^\vee_1)^{\oplus n}}(t)=\frac{(t^3+11t^2 + 11t+1)(t+1)^n}{(1-t)^{n+3}}$,
% and the statement follows from Lemma~\ref{lemma:CLrelation}.

This equation follows from 
\begin{equation}
  \mathrm{Ehr}_{A^\vee_3\oplus (A^\vee_1)^{\oplus n}}(t) = \frac{3}{n+3} \left( 2t\frac{d}{dt}\mathrm{Ehr}_{(A^\vee_1)^{\oplus (n+2)}}(t) 
                           + \mathrm{Ehr}_{(A^\vee_1)^{\oplus (n+2)}}(t) \right)
                         + \frac{n}{n+3} \mathrm{Ehr}_{(A^\vee_1)^{\oplus (n+1)}}(t).
\label{eqn:2ehr}
\end{equation}
as in Theorem~\ref{thm:2111}, and then the statement follows from Lemma~\ref{lemma:CLrelation}.

The equation \eqref{eqn:2ehr} is verified as follows:
\begin{align}
\text{RHS of \eqref{eqn:2ehr}} =& 
\frac{3}{n+3}\bigg(
2t \frac{d}{dt} \frac{(1+t)^{n+2}}{(1-t)^{n+3}} + \frac{(1+t)^{n+2}}{(1-t)^{n+3}} \bigg)
+ \frac{n}{n+4} \frac{(1+t)^{n+1}}{(1-t)^{n+2}} \nonumber \\
=&
\frac{(1+t)^{n+1}}{(1-t)^{n+3}}\bigg(
2t\frac{3}{n+3}\big( (n+2)(1-t) + (1+t)(n+3) \big)  \nonumber \\
& \hspace{50mm} + \frac{3(1+t)(1-t)}{n+3}
+ \frac{n(1-t)^2}{n+3} \bigg) \nonumber \\
=& \frac{(1+t)^{n+1}(1+10t+t^2)}{(1-t)^{n+4}}=\frac{(1+t)^n(1+11t+11t^2+t^3)}{(1-t)^{n+4}} = \textrm{Ehr}_{A^\vee_3\oplus (A^\vee_1)^{\oplus n}}(t). \nonumber
\end{align}

\end{proof}

\begin{remark}
Other than Proposition~\ref{thm:1111}, Theorems~\ref{thm:2111} and \ref{thm:3111}, $A^\vee_4\oplus (A^\vee_1)^{\oplus n}$ and $A^\vee_5\oplus (A^\vee_1)^{\oplus n}$
also seem to be CL by computer calculations for small $n$'s.
On the other hand, also from observation by computer calculation for small $n$'s,
$A^\vee_m\oplus (A^\vee_1)^{\oplus n}$ is not CL for $m\ge 7$ and $n\ge 2$.
The behavior of $A^\vee_6\oplus (A^\vee_1)^{\oplus n}$ is somewhat strange so that it is CL for odd $n$'s and not CL for even $n$'s. 
\end{remark}

\begin{remark}
In the proof of Theorems~\ref{thm:2111} and \ref{thm:3111}, the keys are the equations \eqref{eq:eq1} and \eqref{eq:eq2}.
Analogously, there are other relations among Ehrhart polynomials of $A^\vee_d$'s.
We have found the following equations, though we do not currently find any application.

\begin{multline}
(a) \quad
   E_{A^\vee_3\oplus (A^\vee_2)^{\oplus n}}(k) = \frac{2}{2n+3}E_{(A^\vee_2)^{\oplus (n+1)}}(k)\cdot(2k+1) + \frac{2n+1}{2n+3}E_{A^\vee_1\oplus (A^\vee_2)^{\oplus n}}(k)  
\nonumber
\end{multline}
\vspace{-3mm}
\begin{multline}
(b) \quad
   E_{A^\vee_3\oplus (A^\vee_1)^{\oplus n}}(k) = \frac{2}{n+3}E_{A^\vee_2\oplus (A^\vee_1)^{\oplus n}}(k)\cdot(2k+1) + \frac{2n+1}{n+3}E_{(A^\vee_1)^{\oplus (n+1)}}(k) \\
                                                                                                                      - \frac{n}{n+3}E_{(A^\vee_1)^{\oplus (n-1)}}(k)  
\nonumber
\end{multline}
\vspace{-3mm}
\begin{multline}
(c) \quad
   E_{A^\vee_4\oplus (A^\vee_1)^{\oplus n}}(k) = \frac{5}{2n+8}E_{A^\vee_3\oplus (A^\vee_1)^{\oplus n}}(k)\cdot(2k+1)  + \frac{5(4n+2)}{3(2n+8)}E_{A^\vee_2\oplus (A^\vee_1)^{\oplus n}}(k) \\
                                                                                                                        - \frac{14n+1}{3(2n+8)}E_{(A^\vee_1)^{\oplus n}}(k) 
\nonumber
\end{multline}
\vspace{-3mm}
\begin{multline}
(d) \quad
   E_{(A^\vee_2)^{\oplus 2}\oplus (A^\vee_1)^{\oplus n}}(k) = \frac{3}{2n+8}E_{A^\vee_2\oplus (A^\vee_1)^{\oplus (n+1)}}(k)\cdot(2k+1) + \frac{2n+3}{2n+8}E_{A^\vee_2\oplus (A^\vee_1)^{\oplus n}}(k) \\
                                                                                                                        +\frac{2}{2n+8}E_{(A^\vee_1)^{\oplus n}}(k) 
\nonumber
\end{multline}
\vspace{-3mm}
\begin{multline}
(e) \quad
  E_{A^\vee_3\oplus (A^\vee_1)^{\oplus n}}(k) = \frac{2}{n+3}E_{A^\vee_2\oplus (A^\vee_1)^{\oplus n}}(k)\cdot(2k+1) \\
                     + \frac{n+1}{n+3}\left(\frac{2n+1}{n+1}E_{(A^\vee_1)^{\oplus (n+1)}}(k) - \frac{n}{n+1}E_{(A^\vee_1)^{\oplus (n-1)}}(k) \right)
\nonumber
\end{multline}
\vspace{-3mm}
\begin{multline}
\phantom{(a)}\quad
  \frac{2n+1}{n+1}E_{(A^\vee_1)^{\oplus (n+1)}}(k) - \frac{n}{n+1}E_{(A^\vee_1)^{\oplus (n-1)}}(k)
                    = \frac{2n+1}{(n+1)^2} E_{(A^\vee_1)^{\oplus n}}(k) \cdot (2k+1)  \\
                     + \frac{n^2}{(n+1)^2} E_{(A^\vee_1)^{\oplus (n-1)}}(k)
\nonumber
\end{multline}
\vspace{-3mm}
\begin{multline}
(f) \quad
  E_{A^\vee_4\oplus (A^\vee_1)^{\oplus n}}(k) = \frac{5}{2n+8}E_{A^\vee_3\oplus (A^\vee_1)^{\oplus n}}(k)\cdot(2k+1) \\
                     + \frac{2n+3}{2n+8}\left(\frac{5(4n+2)}{3(2n+3)}E_{A^\vee_2\oplus (A^\vee_1)^{\oplus n}}(k) 
                     - \frac{14n+1}{3(2n+3)}E_{(A^\vee_1)^{\oplus n}}(k) \right) 
\nonumber
\end{multline}
\vspace{-3mm}
\begin{multline}
\phantom{(a)} \quad
  \frac{5(4n+2)}{3(2n+3)}E_{(A^\vee_2)\oplus (A^\vee_1)^{\oplus n}}(k) - \frac{14n+1}{3(2n+3)}E_{(A^\vee_1)^{\oplus n}}(k)\\
                    = \frac{5(2n+1)}{(2n+3)(n+2)} E_{(A^\vee_1)^{\oplus (n+1)}}(k) \cdot (2k+1) 
                     + \frac{(n-1)(2n-1)}{(2n+3)(n+2)} E_{(A^\vee_1)^{\oplus n}}(k) 
\nonumber
\end{multline}
\vspace{-3mm}
\begin{multline}
(f') \quad
  E_{A^\vee_4\oplus (A^\vee_1)^{\oplus n}}(k) = \frac{15}{(2n+8)(n+3)}E_{(A^\vee_1)^{\oplus (n+2)}}(k)\cdot(2k+1)^2 \\
                     + \frac{15(n^2+3n+1)}{2(n+2)(n+3)(n+4)}E_{(A^\vee_1)^{\oplus (n+1)}}(k)\cdot(2k+1) 
                     + \frac{(2n-1)(n-1)}{2(n+2)(n+4)}E_{(A^\vee_1)^{\oplus n}}(k) 
\nonumber
\end{multline}
\end{remark}

\bigskip
%%%%%%%%%%%%%%%%%%%%%%%%%%%%%%%%%%%%%%%%%%%%%%%%%%%%%%%%%%%%%%%%%%%%%%%%%%%%%%%%%%%%%%%%%%%%%%%%%%%%%%%%%%%%%%%%%%%%%%%%%%%%%%%%%%%%%%%%
%%%%%%%%%%%%%%%%%%%%%%%%%%%%%%%%%%%%%%%%%%%%%%%%%%%%%%%%%%%%%%%%%%%%%%%%%%%%%%%%%%%%%%%%%%%%%%%%%%%%%%%%%%%%%%%%%%%%%%%%%%%%%%%%%%%%%%%%
%%%%%%%%%%%%%%%%%%%%%%%%%%%%%%%%%%%%%%%%%%%%%%%%%%%%%%%%%%%%%%%%%%%%%%%%%%%%%%%%%%%%%%%%%%%%%%%%%%%%%%%%%%%%%%%%%%%%%%%%%%%%%%%%%%%%%%%%
%%%%%%%%%%%%%%%%%%%%%%%%%%%%%%%%%%%%%%%%%%%%%%%%%%%%%%%%%%%%%%%%%%%%%%%%%%%%%%%%%%%%%%%%%%%%%%%%%%%%%%%%%%%%%%%%%%%%%%%%%%%%%%%%%%%%%%%%
%%%%%%%%%%%%%%%%%%%%%%%%%%%%%%%%%%%%%%%%%%%%%%%%%%%%%%%%%%%%%%%%%%%%%%%%%%%%%%%%%%%%%%%%%%%%%%%%%%%%%%%%%%%%%%%%%%%%%%%%%%%%%%%%%%%%%%%%

\section{Free sums of $A_d$'s}
In the previous sections, we have studied the root distributions of the Ehrhart polynomials of the free sums of $A^\vee_d$'s.
It is also of interest in studying the free sums of other reflexive polytopes.
For example, how about the free sums of the classical root polytopes $A_d$'s?
Note that since $A_1=A_1^\vee$, the CL-ness and the $R$-interlacing property for $A_1^{\oplus n}=\mathrm{Cr}_n$ also hold. %by \cite{HKM:2017}. 

For the root polytope of type A, the Ehrhart polynomial and the $\delta$-polynomial
known to be as follows (\cite[Theorem~1]{BDV:1999}, \cite[Theorem~2]{ABHPS:2011}):
$$ 
E_{A_d}(k) = \sum_{j=0}^d \binom{d}{j}^2 \binom{k+d-j}{d}, \quad \delta_{A_d}(t) = \sum_{j=0}^d \binom{d}{j}^2 t^j.
$$

We have the following several analogous results. 

\begin{thm}\label{thm:A1n}
$A_1\oplus A_n$ is CL for any $n\ge 1$.
\end{thm}
\begin{proof}
We have the following equality:
\begin{equation}\label{eq:A_1A_n}
  E_{A_1\oplus A_n}(k) = \frac{1}{n+1}E_{A_n}(k)\cdot(2k+1) + \frac{n}{n+1}E_{A_{n-1}}(k).
\end{equation}

This relation follows from the following relation of the Ehrhart series:
\begin{equation}
  \textrm{Ehr}_{A_1\oplus A_n}(t) = 
  \frac{1}{n+1} \left(
     2\frac{d}{dt}\textrm{Ehr}_{A_n}(t)  + \textrm{Ehr}_{A_n}(t) \right)
   + \frac{n}{n+1}\textrm{Ehr}_{A_{n-1}}(t) ,
\label{eqn:rel-EhrA1n}
\end{equation}
which is verified as follows.
Since we have
$$
\textrm{Ehr}_{A_n} = \frac{ \sum_{j=0}^{n} \binom{n}{j}^2 t^j}{(1-t)^{n+1}},
\quad
\textrm{Ehr}_{A_1\oplus A_n} = \frac{ (1+t) \sum_{j=0}^{n} \binom{n}{j}^2 t^j}{(1-t)^{n+1}},
$$
the equation \eqref{eqn:rel-EhrA1n} is equivalent to 
$$
\frac{ (1+t) \sum_{i=0}^{n} \binom{n}{i}^2 }{(1-t)^{n+1}}
=
  \frac{1}{n+1} \left(
     2\frac{d}{dt} \frac{\sum_{j=0}^{n} \binom{n}{j}^2 }{(1-t)^{n+1}} 
     + \frac{\sum_{j=0}^{n} \binom{n}{j}^2 }{(1-t)^{n+1}} \right)
   + \frac{n}{n+1} \frac{\sum_{j=0}^{n-1} \binom{n-1}{j}^2 }{(1-t)^{n}}, 
$$
and we have
\begin{align}
(1+t)\sum_{j=0}^{n} \binom{n}{j}^2 t^j
=&
\frac{2t}{n+1} (1-t) \sum_{j=1}^{n}\binom{n}{j}^2t^{j-1}
+
2t \sum_{j=0}^{n}\binom{n}{j}^2 t^n \nonumber \\
&+
\frac{1}{n+1} (1-t) \sum_{j=0}^{n} \binom{n}{j}^2 t^n
+
\frac{n}{n+1} (1-t)^2 \sum_{j=0}^{n-1} \binom{n-1}{j}^2 t^j . \nonumber
\end{align}
By comparing the coefficients of $t^i$, what we have to show is
\begin{align}
\binom{n}{i}^2 &+ \binom{n}{i-1}^2 
=
\frac{2}{n+1}(i \binom{n}{i}^2 - (i-1)\binom{n}{i-1}^2) + 2\binom{n}{i-1}^2 
\nonumber \\
&+\frac{1}{n+1}(\binom{n}{i}^2 - \binom{n}{i-1}^2) 
+\frac{n}{n+1}(\binom{n-1}{i}^2 - 2\binom{n-1}{i-1}^2 + \binom{n-1}{i-2}^2), 
\label{eqn:coco}
\end{align}
where $\binom{n}{i}$ is assumed to be 0 when $i<0$ or $i>n$.
This is verified by
$$
\binom{n}{i}^2 + \binom{n}{i-1}^2
=\binom{n}{i}^2 + \frac{i^2}{(n-i+1)^2} \binom{n}{i}^2 = \frac{n^2-2in+2n+2i^2-2i+1}{(n-i+1)^2}\binom{n}{i}^2
$$
and 
\begin{align}
\text{RHS of \eqref{eqn:coco}}
=& 
\frac{2}{n+1} (i\binom{n}{i}^2 - (i-1)\binom{n}{i-1}^2) + 2\binom{n}{i-1}^2 +
\frac{1}{n+1} (\binom{n}{i}^2 - \binom{n}{i-1}^2) \nonumber \\
&+\frac{n}{n+1} (\frac{(n-i)^2}{n^2} \binom{n}{i}^2 
 - 2 \frac{(n-i+1)^2}{n^2} \binom{n}{i-1}^2 
 + \frac{(i-1)^2}{n^2} \binom{n}{i-1}^2 ) \nonumber \\
=&
\frac{n^2+n+i^2}{n(n+1)} \binom{n}{i}^2 + \frac{2ni-n-i^2+2i-1}{n(n+1)} \binom{n}{i-1}^2 \nonumber \\
=& \frac{n^2-2in+2n+2i^2-2i+1}{(n-i+1)^2} \binom{n}{i}^2 . \nonumber
\end{align}

Since the Ehrhart polynomials of $A_n$ and $A_{n-1}$ are $R$-interlacing as shown in \cite{HKM:2017}, 
the statement follows from Lemma~\ref{lemma:CLrelation} and \eqref{eq:A_1A_n}.
\end{proof}

\begin{prop} \label{thm:A111-211}
%$A_1^{\oplus n}$ and 
$A_2\oplus A_1^{\oplus n}$ are CL for any $n\ge 1$.
\end{prop}
\begin{proof}
This follows from %Proposition~\ref{thm:1111} and 
Theorem~\ref{thm:2111}, since we have $E_{A_1}(k)=E_{A^\vee_1}(k)$ and $E_{A_2}(k)=E_{A^\vee_2}(k)$.
%(In fact, we have $A_1^{\oplus n}\simeq (A_1^\vee)^{\oplus n}\simeq \mathrm{Cr}_n$ and $A_2\simeq A_2^\vee$.)
\end{proof}

\begin{thm} \label{thm:A3111}
$A_3\oplus A_1^{\oplus n}$ is CL for any $n\ge 1$.
\end{thm}
\begin{proof}
We have the following relation:
\begin{equation}
  E_{A_3\oplus A_1^{\oplus n}}(k) = \frac{5}{2(n+3)}E_{\mathrm{Cr}_{n+2}}(k)\cdot(2k+1) + \frac{2n+1}{2(n+3)}E_{\mathrm{Cr}_{n+1}}(k).
  \label{eqn:thm5-1}
\end{equation}
This follows from the relation of the Ehrhart series:
\begin{equation}
\mathrm{Ehr}_{A_3\oplus A_1^{\oplus n}} = \frac{5}{2(n+3)}\left(
2t \frac{d}{dt} \mathrm{Ehr}_{\mathrm{Cr}_{n+2}}(t) + \mathrm{Ehr}_{\mathrm{Cr}_{n+2}}(t) \right)
+ \frac{2n+1}{2(n+3)}\mathrm{Ehr}_{\mathrm{Cr}_{n+1}}(t).
\label{eqn:thm5-2}
\end{equation}
The equation \eqref{eqn:thm5-1} is derived by comparing the coefficients of $t^k$ in \eqref{eqn:thm5-2}.
The equation \eqref{eqn:thm5-2} can be verified using
$\mathrm{Ehr}_{A_3\oplus A_1^{\oplus n}}(t)=\frac{(1+9t+9t^2+t^3)(1+t)^n}{(1-t)^{n+4}}$
and $\mathrm{Ehr}_{\mathrm{Cr}_n}(t)=\frac{(1+t)^n}{(1-t)^{n+1}}$ as follows:
\begin{align}
\text{RHS of \eqref{eqn:thm5-2}} =& 
\frac{5}{2(n+3)}\bigg(
2t \frac{d}{dt} \frac{(1+t)^{n+2}}{(1-t)^{n+3}} + \frac{(1+t)^{n+2}}{(1-t)^{n+3}} \bigg)
+ \frac{2n+1}{2(n+3)} \frac{(1+t)^{n+1}}{(1-t)^{n+2}} \nonumber \\
=&
\frac{(1+t)^n}{(1-t)^{n+4}}\bigg(
2t\frac{5}{2(n+3)}\big( (n+2)(1+t)(1-t) + (1+t)^2(n+3) \big)  \nonumber \\
& \hspace{40mm} + \frac{5(1-t)(1+t)^2}{2(n+3)}
+ \frac{(2n+1)(1+t)(1-t)^2}{2(n+3)} \bigg) \nonumber \\
=& \frac{(1+t)^n(1+9t+9t^2+t^3)}{(1-t)^{n+4}} = \textrm{Ehr}_{A_3\oplus A_1^{\oplus n}}(t). \nonumber
\end{align}

The $R$-interlacing property follows from Lemma~\ref{lemma:CLrelation},
since the Ehrhart polynomials of the cross polytopes $\mathrm{Cr}_{n+1}$ and $\mathrm{Cr}_n$ are $R$-interlacing.
\end{proof}

% The above results are analogous to the results in previous sections. 
% Especially, %Theorems \ref{thm:A1111} and 
% Theorem~\ref{thm:A2111} is derived from the relation with the same coefficients as Theorem~\ref{thm:2111}. 
% On the other hand, the relation used in Theorem~\ref{thm:A3111} has different coefficients from Theorem~\ref{thm:3111}. 

\bigskip
\newcommand{\CL}{C}
\newcommand{\NC}{n}
Table~\ref{table:Amn} shows the CL-ness of $A_m \oplus A_n$, calculated by computer using Pari/GP. 
Comparing with that of $A^\vee_m \oplus A^\vee_n$, the behavior is somewhat complex. (Here, ``\CL'' means CL, and ``\NC'' means not CL.) 
Similar to the case of $A_m \oplus A_n$, it is CL for small $m$ and $n$.
On the other hand, the behavior looks different when $m$ and $n$ are large.

\begin{table}[htb]
\caption{CL-ness of $A_m \oplus A_n$ with $m,n\le 20$}
\label{table:Amn}
\begin{center}
\footnotesize
\begin{tabular}{| l | l | l | l | l | l | l | l | l | l | l | l | l | l | l | l | l | l | l | l | l | l |} \hline
$n$ $\backslash$ $m$ &
       0 &   1 &   2 &   3 &   4 &   5 &   6 &   7 &   8 &   9 &  10 &  11 &  12 &  13 &  14 &  15 &  16 &  17 &  18 &  19 &  20 \\ \hline
 0 & \CL & \CL & \CL & \CL & \CL & \CL & \CL & \CL & \CL & \CL & \CL & \CL & \CL & \CL & \CL & \CL & \CL & \CL & \CL & \CL & \CL \\ \hline 
 1 & \CL & \CL & \CL & \CL & \CL & \CL & \CL & \CL & \CL & \CL & \CL & \CL & \CL & \CL & \CL & \CL & \CL & \CL & \CL & \CL & \CL \\ \hline
 2 & \CL & \CL & \CL & \CL & \CL & \CL & \CL & \CL & \CL & \CL & \CL & \CL & \CL & \CL & \NC & \NC & \NC & \NC & \NC & \NC & \NC \\ \hline
 3 & \CL & \CL & \CL & \CL & \CL & \CL & \CL & \CL & \CL & \NC & \NC & \NC & \NC & \NC & \NC & \NC & \NC & \NC & \NC & \NC & \CL \\ \hline
 4 & \CL & \CL & \CL & \CL & \CL & \CL & \CL & \CL & \CL & \NC & \NC & \NC & \NC & \NC & \NC & \CL & \CL & \CL & \CL & \NC & \NC \\ \hline 
 5 & \CL & \CL & \CL & \CL & \CL & \CL & \CL & \CL & \CL & \NC & \NC & \NC & \NC & \NC & \CL & \CL & \CL & \NC & \NC & \NC & \NC \\ \hline
 6 & \CL & \CL & \CL & \CL & \CL & \CL & \CL & \CL & \CL & \CL & \NC & \NC & \NC & \NC & \CL & \CL & \CL & \NC & \NC & \NC & \CL \\ \hline
 7 & \CL & \CL & \CL & \CL & \CL & \CL & \CL & \CL & \CL & \CL & \CL & \NC & \NC & \NC & \CL & \CL & \CL & \NC & \NC & \CL & \NC \\ \hline
 8 & \CL & \CL & \CL & \CL & \CL & \CL & \CL & \CL & \CL & \CL & \CL & \CL & \NC & \NC & \CL & \CL & \CL & \NC & \NC & \NC & \NC \\ \hline
 9 & \CL & \CL & \CL & \NC & \NC & \NC & \CL & \CL & \CL & \CL & \CL & \CL & \NC & \NC & \NC & \CL & \CL & \CL & \NC & \NC & \NC \\ \hline
10 & \CL & \CL & \CL & \NC & \NC & \NC & \NC & \CL & \CL & \CL & \CL & \CL & \CL & \NC & \NC & \NC & \CL & \CL & \NC & \NC & \NC \\ \hline
11 & \CL & \CL & \CL & \NC & \NC & \NC & \NC & \NC & \CL & \CL & \CL & \CL & \CL & \CL & \NC & \NC & \CL & \CL & \CL & \NC & \NC \\ \hline
12 & \CL & \CL & \CL & \NC & \NC & \NC & \NC & \NC & \NC & \NC & \CL & \CL & \CL & \CL & \CL & \NC & \NC & \CL & \CL & \CL & \NC \\ \hline
13 & \CL & \CL & \CL & \NC & \NC & \NC & \NC & \NC & \NC & \NC & \NC & \CL & \CL & \CL & \CL & \CL & \NC & \NC & \CL & \CL & \NC \\ \hline
14 & \CL & \CL & \NC & \NC & \NC & \CL & \CL & \CL & \CL & \NC & \NC & \NC & \CL & \CL & \CL & \CL & \CL & \NC & \NC & \CL & \CL \\ \hline
15 & \CL & \CL & \NC & \NC & \CL & \CL & \CL & \CL & \CL & \CL & \NC & \NC & \NC & \CL & \CL & \CL & \CL & \CL & \NC & \NC & \CL \\ \hline
16 & \CL & \CL & \NC & \NC & \CL & \CL & \CL & \CL & \CL & \CL & \CL & \CL & \NC & \NC & \CL & \CL & \CL & \CL & \CL & \NC & \NC \\ \hline
17 & \CL & \CL & \NC & \NC & \CL & \NC & \NC & \NC & \NC & \CL & \CL & \CL & \CL & \NC & \NC & \CL & \CL & \CL & \CL & \CL & \NC \\ \hline
18 & \CL & \CL & \NC & \NC & \CL & \NC & \NC & \NC & \NC & \NC & \NC & \CL & \CL & \CL & \NC & \NC & \CL & \CL & \CL & \CL & \CL \\ \hline
19 & \CL & \CL & \NC & \NC & \NC & \NC & \NC & \CL & \NC & \NC & \NC & \NC & \CL & \CL & \CL & \NC & \NC & \CL & \CL & \CL & \CL \\ \hline
20 & \CL & \CL & \NC & \CL & \NC & \NC & \CL & \NC & \NC & \NC & \NC & \NC & \NC & \NC & \CL & \CL & \NC & \NC & \CL & \CL & \CL \\ \hline
\end{tabular}
\end{center}
\end{table}

\bigskip
%%%%%%%%%%%%%%%%%%%%%%%%%%%%%%%%%%%%%%%%%%%%%%%%%%%%%%%%%%%%%%%%%%%%%%%%%%%%%%%%%%%%%%%%%%%%%%%%%%%%%%%%%%%%%%%%%%%%%%%%%%%%%%%%%%%%%%%%
%%%%%%%%%%%%%%%%%%%%%%%%%%%%%%%%%%%%%%%%%%%%%%%%%%%%%%%%%%%%%%%%%%%%%%%%%%%%%%%%%%%%%%%%%%%%%%%%%%%%%%%%%%%%%%%%%%%%%%%%%%%%%%%%%%%%%%%%
%%%%%%%%%%%%%%%%%%%%%%%%%%%%%%%%%%%%%%%%%%%%%%%%%%%%%%%%%%%%%%%%%%%%%%%%%%%%%%%%%%%%%%%%%%%%%%%%%%%%%%%%%%%%%%%%%%%%%%%%%%%%%%%%%%%%%%%%
%%%%%%%%%%%%%%%%%%%%%%%%%%%%%%%%%%%%%%%%%%%%%%%%%%%%%%%%%%%%%%%%%%%%%%%%%%%%%%%%%%%%%%%%%%%%%%%%%%%%%%%%%%%%%%%%%%%%%%%%%%%%%%%%%%%%%%%%
%%%%%%%%%%%%%%%%%%%%%%%%%%%%%%%%%%%%%%%%%%%%%%%%%%%%%%%%%%%%%%%%%%%%%%%%%%%%%%%%%%%%%%%%%%%%%%%%%%%%%%%%%%%%%%%%%%%%%%%%%%%%%%%%%%%%%%%%

\section*{Appendix: Some results by computer calculation}

As mentioned in Section~\ref{sec:mn},
we investigated the CL-ness for $A^\vee_m\oplus A^\vee_n$ with $m,n\ge 2$ by computer calculations. 
The $\delta$-vectors, Ehrhart polynomials, and the CL-ness are listed in Table~\ref{table:mn-CLness}.
The computation is done by using {\tt Pari/GP}. 
The results are summarized in Theorem~\ref{thm:rank2cp}.

\bigskip

We also calculated the case of the free sums of three or four $A^\vee_d$'s with small parameters (up to $20$) by using {\tt Pari/GP}. 
For the case of the free sums of three $A^\vee_d$, our computer calculations are as follows. 

\bigskip
\noindent
$A^\vee_1\oplus A^\vee_1\oplus A^\vee_n$ : CL for $n=1, \dots, 5$, not CL for $n=6,\dots, 20$ \\
$A^\vee_1\oplus A^\vee_2\oplus A^\vee_n$ : CL for $n=2, \dots, 4$, not CL for $n=5,\dots, 20$ \\
$A^\vee_1\oplus A^\vee_3\oplus A^\vee_n$ : CL for $n=3,4$, not CL for $n=5, \dots, 20$ \\
$A^\vee_1\oplus A^\vee_4\oplus A^\vee_n$ : not CL for $n=4$, CL for $n=5$, not CL for $n=6,\dots,20$ \\
$A^\vee_2\oplus A^\vee_2\oplus A^\vee_n$ : CL for $n=2,3$, not CL for $n=4,\dots,20$ \\
$A^\vee_2\oplus A^\vee_3\oplus A^\vee_n$ : CL for $n=3,\dots,6$, not CL for $n=7,\dots,20$ \\[2mm]
Other parameters (each from 1 to 20) not listed here are not CL, up to permutation of the parameters.
%(Remark that permuting the parameters does not affect the polynomials.)

\bigskip
Similarly to the case of two parameters, roughly speaking, we can observe that CL-ness hods for small parameters and CL-ness does not hold for large parameters.
However, the boundary is sometimes complexified such that the CL/nonCL is not monotone: 
% $A^\vee_{1,4,n}$ is CL for small $n\le 3$, not CL for $n=5$, CL for $n=5$, and not CL for $n=6,\dots,20$.
% The same is observed for $A^\vee_{1,m,5}$. It is CL for $m=2$, not CL for $m=2,3$, CL for $m=4$, and not CL for $m\ge 5$.
$A^\vee_1\oplus A^\vee_4\oplus A^\vee_n$ is CL for small $n\le 3$, not CL for $n=4$, CL for $n=5$, and not CL for $n=6,\dots,20$.
The same can be observed for $A^\vee_1\oplus A^\vee_m\oplus A^\vee_5$. It is not CL for $m=2,3$, CL for $m=4$, and not CL for $m\ge 5$.

For the case of four $A^\vee_d$'s, our computer calculations are shown as follows.

\bigskip
\noindent
$A^\vee_1\oplus A^\vee_1\oplus A^\vee_1\oplus A^\vee_n$ : CL for $n=1, \dots, 4$, not CL for $n=5$, CL for $n=6$, not CL for $n=7,\dots, 20$ \\
$A^\vee_1\oplus A^\vee_1\oplus A^\vee_2\oplus A^\vee_n$ : CL for $n=2, \dots, 5$, not CL for $n=6,\dots, 20$ \\
$A^\vee_1\oplus A^\vee_1\oplus A^\vee_3\oplus A^\vee_n$ : CL for $n=3, \dots, 5$, not CL for $n=6, \dots, 20$ \\
$A^\vee_1\oplus A^\vee_1\oplus A^\vee_4\oplus A^\vee_n$ : CL for $n=4$, not CL for $n=5,\dots,20$ \\
%$A^\vee_1\oplus A^\vee_1\oplus A^\vee_5\oplus A^\vee_n$ : not CL for $n=5,\dots, 20$ \\
$A^\vee_1\oplus A^\vee_2\oplus A^\vee_2\oplus A^\vee_n$ : CL for $n=2,\dots, 6$, not CL for $n=7,\dots,20$ \\
$A^\vee_1\oplus A^\vee_2\oplus A^\vee_3\oplus A^\vee_n$ : CL for $n=3,\dots,5$, not CL for $n=6,\dots,20$ \\
$A^\vee_1\oplus A^\vee_2\oplus A^\vee_4\oplus A^\vee_n$ : CL for $n=4$, not CL for $n=5$, CL for $n=6$, not CL for $n=7,\dots,20$ \\
%$A^\vee_1\oplus A^\vee_2\oplus A^\vee_5\oplus A^\vee_n$ : not CL for $n=5,\dots,20$ \\
$A^\vee_1\oplus A^\vee_3\oplus A^\vee_3\oplus A^\vee_n$ : not CL for $n=3$, CL for $n=4$, not CL for $n=5,\dots,20$ \\
$A^\vee_1\oplus A^\vee_3\oplus A^\vee_4\oplus A^\vee_n$ : not CL for $n=4$, CL for $n=5$, not CL for $n=6,\dots,20$ \\
%$A^\vee_2\oplus A^\vee_2\oplus A^\vee_2\oplus A^\vee_n$ : not CL for $n=2,\dots,20$ \\
$A^\vee_2\oplus A^\vee_2\oplus A^\vee_3\oplus A^\vee_n$ : CL for $n=3$, not CL for $n=4,\dots, 20$ \\
$A^\vee_2\oplus A^\vee_3\oplus A^\vee_3\oplus A^\vee_n$ : not CL for $n=3$, CL for $n=4$, not CL for $n=5,\dots,20$ \\[2mm]
Other parameters (each from 1 to 20) not listed here are not CL,
up to permutation of the parameters.

% \eject
% \vspace*{-35mm}\hspace*{-25mm}
\begin{table}[p]
% \vspace*{-15mm}
\caption{$\delta$-vectors, Ehrhart polynomials, and CL-ness of $A^\vee_m\oplus A^\vee_n$ with $2\le m,n\le 7$}
\label{table:mn-CLness}
\hspace*{-20mm}
\rotatebox{0}{
% \begin{minipage}{150mm}
\begin{tabular}{| l | l | l | l |} \hline
% $n$ & $m$ & $h(A^\vee_{n,m})$,  $E_{A^\vee_{n,m}}^\text{eq}(k)$ &  \\ \hline
$m$ & $n$ & $\delta(A^\vee_m \oplus A^\vee_n)$,  $E_{A^\vee_m\oplus A^\vee_n}(k)$ &  \\ \hline
2 & 2 & 
\begin{tabular}{l}
\small
(1,8,18,8,1) \\ 
$\frac{3}{2}x^4 + 3x^3 + \frac{9}{2}x^2 + 3x + 1$ 
\end{tabular}
& CL \\ \hline
2 & 3 & 
\begin{tabular}{l}
\small
(1,15,56,56,15,1) \\ 
$\frac{6}{5}x^5 + 3x^4 + 6x^3 + 6x^2 + \frac{19}{5}x + 1$ 
\end{tabular}
& CL \\ \hline
2 & 4 & 
\begin{tabular}{l}
\small
(1,30,171,316,171,30,1) \\ 
$x^6 + 3x^5 + \frac{15}{2}x^4  + 10x^3 + \frac{19}{2}x^2 + 5x + 1$ 
\end{tabular}
& CL \\ \hline
2 & 5 & 
\begin{tabular}{l}
\small
(1,61,531,1567,1567,531,61,1) \\ 
$\frac{6}{7}x^7 + 3x^6 + 9x^5 + 15x^4 + 19x^3 + 15x^2 + \frac{43}{7}x + 1$ 
\end{tabular}
& CL \\ \hline
2 & 6 & 
\begin{tabular}{l}
\small
(1,124,1672,7300,12046,7300,1672,124,1) \\ 
$\frac{3}{4}x^8 + 3x^7 + \frac{21}{2}x^6 + 21x^5 + \frac{133}{4}x^4 + 35x^3 + \frac{43}{2}x^2 + 7x + 1$  
\end{tabular}
& not CL \\ \hline
2 & 7 & 
\begin{tabular}{l}
\small
(1,251,5282,33038,82388,82388,33038,5282,251,1) \\ 
$\frac{2}{3}x^9 + 3x^8 + 12x^7 + 28x^6 + \frac{266}{5}x^5 + 70x^4 + \frac{172}{3}x^3 + 28x^2 + \frac{39}{5}x+1$ 
\end{tabular}
& not CL\\ \hline
3 & 3 & 
\begin{tabular}{l}
\small
(1,22,143,244,143,22,1) \\ 
$\frac{4}{5}x^6 + \frac{12}{5}x^5 + 6x^4 + 8x^3 + \frac{36}{5}x^2 + \frac{18}{5}x+1$  
\end{tabular}
& CL \\ \hline
3 & 4 & 
\begin{tabular}{l}
\small
(1,37,363,1039,1039,363,37,1) \\ 
$\frac{4}{7}x^7 + 2x^6 + 6x^5 + 10x^4 + 12x^3 + 9x^2 + \frac{31}{7}x + 1$ 
\end{tabular}
& CL \\ \hline
3 & 5 & 
\begin{tabular}{l}
\small
(1,68,940,4252,6758,4252,940,68,1) \\ 
$\frac{3}{7}x^8 + \frac{12}{7}x^7 + 6x^6 + 12x^5 + 18x^4 + 18x^3 + \frac{95}{7}x^2 + \frac{44}{7}x + 1$ 
\end{tabular}
& not CL\\ \hline
3 & 6 & 
\begin{tabular}{l}
\small
(1,131,2522,16838,40988,40988,16838,2522,131,1) \\ 
$\frac{1}{3}x^9 + \frac{3}{2}x^8 + 6x^7 + 14x^6 + \frac{126}{5}x^5 + \frac{63}{2}x^4 + \frac{95}{3}x^3 + 22x^2 + \frac{39}{5}x + 1$ 
\end{tabular}
& not CL\\ \hline
3 & 7 & 
\begin{tabular}{l}
\small
(1,258,7021,65560,234898,352204,234898,65560,7021,258,1) \\ 
$\frac{4}{15}x^{10} + \frac{4}{3}x^9 + 6x^8 + 16x^7 + \frac{168}{5}x^6 + \frac{252}{5}x^5 + \frac{190}{3}x^4 + \frac{176}{3}x^3 + \frac{154}{5}x^2 + \frac{38}{5}x + 1$ 
\end{tabular}
& not CL\\ \hline
4 & 4 & 
\begin{tabular}{l}
\small
(1,52,808,3484,5710,3484,808,52,1) \\ 
$\frac{5}{14}x^8 + \frac{10}{7}x^7 + 5x^6 + 10x^5 + 15x^4 + 15x^3 + \frac{135}{14}x^2 + \frac{25}{7}x + 1$ 
\end{tabular}
& not CL \\ \hline
4 & 5 & 
\begin{tabular}{l}
\small
(1,83,1850,11942,29324,29324,11942,1850,83,1) \\ 
$\frac{5}{21}x^9 + \frac{15}{14}x^8 + \frac{30}{7}x^7 + 10x^6 + 18x^5 + \frac{45}{2}x^4 + \frac{415}{21}x^3 + \frac{80}{7}x^2 + \frac{33}{7}x + 1$ 
\end{tabular}
& not CL \\ \hline
4 & 6 & 
\begin{tabular}{l}
\small
(1,146,4377,41328,145734,221628,145734,41328,4377,146,1) \\ 
$\frac{1}{6}x^{10} + \frac{5}{6}x^9 + \frac{15}{4}x^8 + 10x^7 + 21x^6 + \frac{63}{2}x^5 +\frac{415}{12}x^4 + \frac{80}{3}x^3 + \frac{37}{2}x^2 + 9x + 1$ 
\end{tabular}
& not CL \\ \hline
4 & 7 & 
\begin{tabular}{l}
\small
(1,273,10781,143565,711474,1553106,1553106,711474,143565,10781,273,1) \\ 
$\frac{4}{33}x^{11} + \frac{2}{3}x^{10} + \frac{10}{3}x^9 + 10x^8 + 24x^7 + 42x^6 + \frac{166}{3}x^5 + \frac{160}{3}x^4 + \frac{146}{3}x^3 + 35x^2 + \frac{127}{11}x + 1$ 
\end{tabular}
& not CL \\ \hline
5 & 5 & 
\begin{tabular}{l}
\small
(1,114,3853,35032,125746,188908,125746,35032,3853,114,1) \\ 
$\frac{1}{7}x^{10} + \frac{5}{7}x^9 + \frac{45}{14}x^8 + \frac{60}{7}x^7 + 18x^6 + 27x^5 + \frac{425}{14}x^4 + \frac{170}{7}x^3 + \frac{72}{7}x^2 + \frac{10}{7}x + 1$ 
\end{tabular}
& not CL \\ \hline
5 & 6 & 
\begin{tabular}{l}
\small
(1,177,8333,106845,534882,1164162,1164162,534882,106845,8333,177,1) \\ 
$\frac{1}{11}x^{11} + \frac{1}{2}x^{10} + \frac{5}{2}x^9 + \frac{15}{2}x^8 + 18x^7 + \frac{63}{2}x^6 + \frac{85}{2}x^5 + \frac{85}{2}x^4 + 28x^3 + 11x^2 + \frac{43}{11}x + 1$ 
\end{tabular}
& not CL \\ \hline
5 & 7 &
\begin{tabular}{l}
\small
(1,304,18674,335216,2277039,6922080,9923772,6922080,2277039,335216,18674,304,1) \\
$\frac{2}{33}x^{12} + \frac{4}{11}x^{11} + 2x^{10} + \frac{20}{3}x^9 + 18x^8 + 36x^7 + \frac{170}{3}x^6 + 68x^5 + 55x^4 + \frac{82}{3}x^3 + \frac{289}{11}x^2 + \frac{216}{11}x + 1$
\end{tabular}
& not CL \\ \hline
6 & 6 & 
\begin{tabular}{l}
\small
(1,240,16782,290672,2000703,6040992,8702820,6040992,2000703,290672,16782,240,1) \\
$\frac{7}{132}x^{12} + \frac{7}{22}x^{11} + \frac{7}{4}x^{10} + \frac{35}{6}x^9 + \frac{63}{4}x^8 + \frac{63}{2}x^7 + \frac{595}{12}x^6 + \frac{119}{2}x^5 + 56x^4 + \frac{119}{3}x^3 + \frac{63}{22}x^2 - \frac{119}{11}x + 1$
\end{tabular}
& not CL \\ \hline
6 & 7 & 
\begin{tabular}{l}
\small
(1,367,35124,827372,7600805,31146987,61995744,61995744,31146987,7600805,827372,35124,367,1) \\
$\frac{14}{429}x^{13} + \frac{7}{33}x^{12} + \frac{14}{11}x^{11} + \frac{14}{3}x^{10} + 14x^9 + \frac{63}{2}x^8 + \frac{170}{3}x^7 + \frac{238}{3}x^6 + \frac{441}{5}x^5 + \frac{455}{6}x^4 + \frac{357}{11}x^3$ \\ $- \frac{28}{11}x^2 - \frac{1163}{715}x + 1$
\end{tabular}
& not CL \\ \hline
7 & 7 & 
\begin{tabular}{l}
\small
(1,494,69595,2151980,26176873,141829106,380179131,524888040,380179131,141829106,26176873, \\ \small 2151980,69595,494,1) \\
$\frac{8}{429}x^{14} + \frac{56}{429}x^{13} + \frac{28}{33}x^{12} + \frac{112}{33}x^{11} + \frac{56}{5}x^{10} + 28x^9 + \frac{170}{3}x^8 + \frac{272}{3}x^7 + \frac{1736}{15}x^6 + \frac{1736}{15}x^5+ \frac{1260}{11}x^4$ \\ $ + \frac{1120}{11}x^3 - \frac{32184}{715}x^2 - \frac{61306}{715}x + 1$
\end{tabular}
& not CL \\ \hline
\end{tabular}
% \end{minipage}
}
\end{table}

\end{document}